\documentclass[12pt]{article}
\usepackage{amssymb,amsmath,bm}
\usepackage{natbib,bbm}
\usepackage{constants}
\begin{document}
\newcommand{\bea}{\begin{eqnarray}}
\newcommand{\ena}{\end{eqnarray}}
\newcommand{\beas}{\begin{eqnarray*}}
\newcommand{\enas}{\end{eqnarray*}}
\newcommand{\beq}{\begin{equation}}
\newcommand{\enq}{\end{equation}}
\def\qed{\hfill \mbox{\rule{0.5em}{0.5em}}}
\newcommand{\bbox}{\hfill $\Box$}
\newcommand{\ignore}[1]{}
\newcommand{\ignorex}[1]{#1}
\newcommand{\wtilde}[1]{\widetilde{#1}}
\newcommand{\mq}[1]{\mbox{#1}\quad}
\newcommand{\qmq}[1]{\quad\mbox{#1}\quad}
\newcommand{\qm}[1]{\quad\mbox{#1}}
\newcommand{\nn}{\nonumber}
\newcommand{\Bvert}{\left\vert\vphantom{\frac{1}{1}}\right.}
\newcommand{\To}{\rightarrow}
\newcommand{\supp}{\mbox{supp}}

\newtheorem{theorem}{Theorem}[section]
\newtheorem{corollary}{Corollary}[section]
\newtheorem{conjecture}{Conjecture}[section]
\newtheorem{proposition}{Proposition}[section]
\newtheorem{lemma}{Lemma}[section]
\newtheorem{definition}{Definition}[section]
\newtheorem{example}{Example}[section]
\newtheorem{remark}{Remark}[section]
\newtheorem{case}{Case}[section]
\newtheorem{condition}{Condition}[section]
\newcommand{\pf}{\noindent {\it Proof:} }
\newcommand{\proof}{\noindent {\it Proof:} }

\title{{\bf Concentration inequalities via zero bias couplings}}
\author{Larry Goldstein\thanks{work partially supported by NSA
grant H98230-11-1-0162.}  and \"{U}mit I\c{s}lak\\University of
Southern California} \footnotetext{AMS 2000 subject classifications:
Primary 60E15\ignore{inequalities stochastic orderings},
60C05\ignore{Combinatorial probability}, 62G10\ignore{Hypothesis
testing}.} 

\maketitle \vspace{0.25in}
\begin{abstract}
The tails of the distribution of a mean zero, variance $\sigma^2$
random variable $Y$ satisfy concentration of measure inequalities of
the form $\mathbb{P}(Y \ge t) \le \exp(-B(t))$ for
\beas
B(t)=\frac{t^2}{2( \sigma^2 + ct)} \qmq{for
$t \ge 0$, and}
B(t)=\frac{t}{c}\left( \log t - \log \log t  - \frac{\sigma^2}{c}\right)
\quad \mbox{for $t>e$}
\enas
whenever there exists a zero biased coupling of
$Y$ bounded by $c$, under suitable
conditions on the existence of the moment generating function of
$Y$. These inequalities apply in cases where $Y$ is not a function of independent variables, such as
for the Hoeffding statistic $Y=\sum_{i=1}^n a_{i\pi(i)}$
where $A=(a_{ij})_{1 \le i,j \le n} \in \mathbb{R}^{n \times n}$ and
the permutation $\pi$ has the uniform distribution over the
symmetric group, and when its distribution is constant on cycle
type.
\end{abstract}

\section{Introduction}
Since the seminal work of \cite{tal}, the concentration of measure
phenomenon has attracted a great deal of attention of many
researchers working in very diverse fields, see the extensive
treatments of \cite{ledoux} and the recent text of \cite{blm}.
The work of \cite{chatterjee} uncovered
connections between concentration phenomenon and Stein's method, which produces
non-asymptotic error bounds for distributional approximation, see Stein (1972,1986), and also
\cite{goldsteinbook} and \cite{nathan} for overviews. Though
the application of the majority of concentration results requires the
quantity of interest to be a function of independent random
variables, \cite{chatterjee} demonstrated that tail bounds for functions of
dependent random variables, including Hoeffding's statistic and the
net magnetization in the Curie-Weiss model, can be derived using Stein's
exchangeable pair coupling. Use of the size bias
coupling, another important technique from Stein's method, was
shown in Ghosh and Goldstein (2011ab) to produce concentration bounds for the
number of relatively ordered subsequences of a random permutation,
sliding window statistics, the number of local maxima of a random
function on a graph, degrees of random graphs, multinomial
occupation models and coverage problems in stochastic geometry.

Here we focus on the zero bias coupling, also borrowed from Stein's
method and first introduced in \cite{goldsteinreinertzerobias}, and
show how it too may be used to yield concentration bounds. Recall from
\cite{goldsteinreinertzerobias} that for any mean zero random
variable $Y$ with positive, finite variance $\sigma^2$, there exists
a distribution for a random variable $Y^*$ satisfying
\begin{equation}\label{zerobiasdefn}
    \mathbb{E}[Y f(Y)] = \sigma^2 \mathbb{E}[f'(Y^*)]
\end{equation}
for all absolutely continuous functions $f$ for which the
expectation of either side exists; the variable $Y^*$ is said to
have the $Y$-zero biased distribution. A restatement of a result of
\cite{Stein} shows that a mean zero random variable $Y$ is normal if
and only if $Y=_d Y^*.$ Thus, if $Y$ and $Y^*$ can be coupled
closely it is natural to expect that the behavior of $Y$, including
the decay of its tail probabilities, may have behavior similar to that of the
normal. Our main results in the next section justify this heuristic.
For the use of zero bias couplings to produce bounds in normal
approximations see, for instance,  \cite{goldsteinbook} and the references
therein.

Applications of our results will be to the \cite{hoe51} statistic
given by \beas Y=\sum_{i=1}^n a_{i\pi(i)}, \enas depending on an
array $(a_{ij})_{1 \le i,j \le n}$ of real numbers and a random
permutation $\pi$. The quantity $Y$ arises in many applications,
permutation testing foremost among them, see \cite{ww} for a seminal
reference. Our results provide concentration bounds for $Y$ when the
distribution of $\pi$ is uniformly distributed over the symmetric
group, and when its distribution is constant on the cycle type of
$\pi$; for the latter, see \cite{goldsteinrinott} for a statistical
application where $\pi$ is chosen uniformly from the class of fixed
point free involutions.

We present our main result, Theorem \ref{maintheorem}, in Section
\ref{mainresultssection}, applications in Section
\ref{hoeffding} and the proof of Theorem \ref{maintheorem} in
Section \ref{proofs}.

\section{Main Result}\label{mainresultssection}

\begin{theorem}\label{maintheorem}
Let $Y$ be a mean zero random variable with variance $\sigma^2 \in
(0, \infty)$ and moment generating function $m(s)=\mathbb{E}[e^{s
Y}]$, and let $Y^*$ have the $Y$-zero bias distribution and be
defined on the same space as $Y$.

(a). If $Y^*-Y \le c$ for some $c >0$ and $m(s)$ exists for all $s
\in [0,1/c)$, then for all $t \ge 0$
\begin{equation}\label{tailsseparate.2}
\mathbb{P}(Y \geq t) \leq \exp \left(- \frac{t^2}{2(\sigma^2 + ct)}
\right).
\end{equation}
The same upper bound holds for $\mathbb{P}(Y \leq -t)$ if $Y-Y^* \le
c$ when $m(s)$ exists for all $s  \in (-1/c,0]$.
If $|Y^*-Y| \leq c$ for some $c > 0$ and $m(s)$ exists for all $s
\in [0,2/c)$ then for all $t \ge 0$
\begin{equation}\label{tailsseparate}
\mathbb{P}(Y \geq t)\leq \exp \left(- \frac{t^2}{10\sigma^2/3 + ct}
\right),
\end{equation}
with the same upper bound holding for $\mathbb{P}(Y \leq -t)$ if
$m(s)$ exists in $(-2/c,0]$.

(b). If $Y^*-Y \le c$ for some constant $c>0$ and $m(s)$ exists at
$\theta=(\log t - \log \log t)/c$ then for $t > e$
\bea \label{tlogtbound} \mathbb{P}(Y \ge t) \le \exp\left(
-\frac{t}{c}\left( \log t - \log \log t  -
\frac{\sigma^2}{c}\right)\right) \le  \exp\left(
-\frac{t}{2c}\left( \log t -
\frac{2\sigma^2}{c}\right)\right).
\ena

If $Y-Y^* \le c$ then the same bound holds for the left tail
$\mathbb{P}(Y \le -t)$ when $m(-\theta)$ is finite.
\end{theorem}

As regards part (a) and behavior in $n$, we remark that if $|Y^*-Y|
\le c$ and $m(s)$ exists in $[0,2/c)$, the bound
(\ref{tailsseparate.2}) is preferred over (\ref{tailsseparate}) for
$|t| < 4 \sigma^2/ 3c$, a set increasing to $\mathbb{R}$
asymptotically in typical applications where the variance of $Y$
increases to infinity in $n$ while $c$ remains constant. Regarding
behavior in $t$, part (b) of  Theorem \ref{maintheorem} shows that
the respective asymptotic orders as $t \rightarrow \infty$ of $\exp(-t/(2c))$ and $\exp(-t/c)$ of
bounds (\ref{tailsseparate.2}) and (\ref{tailsseparate}), can be improved to $\exp(-t \log t/(2c))$. As
the right tail bound (\ref{tlogtbound}) applies only when $t>e$ it should be
considered as a complementary result to the bounds in (a) that hold
for all $t \ge 0$.

\begin{remark} Theorem 5.1 of \cite{goldsteinbook} states that when $Y$ is a mean zero random variable with variance one for which there exists a coupling to $Y^*$ such that $|Y^*-Y| \le c$ for some $c$ then the Kolmogorov distance between $Y$ and the standard normal distribution is bounded by $2.03c$. Hence for small $c$ the distribution of $Y$ is close to the normal, and in the limiting case where $c$ takes the value
zero inequality (\ref{tailsseparate.2}) is a valid bound when $Y$
has the standard normal distribution. For such $Y$ the inequality of
\cite{chu55} yields \beas \sup_{t \geq 0} \exp (t^2/2) \mathbb{P}(Y
\geq t) = a, \enas for $a=1/2$, while (\ref{tailsseparate.2}) yields
the same bound with $a=1$. The constant $a=1$ also results when
bounding $\mathbb{P}(Y \ge t)$ by $\inf_{s \ge 0}
e^{-st}\mathbb{E}[e^{sY}]$.

On the other hand, again by Theorem 5.1 of \cite{goldsteinbook},
when the distribution of $Y$ is not close to normal there cannot
exist a coupling of $Y$ to $Y^*$ with a small value of $c$, and the
bounds may perform poorly in that the tail decay of $Y$ may in fact
be faster than what is indicated by (\ref{tailsseparate.2}).

\end{remark}

We state some properties of the zero bias distribution. First, from
(\ref{zerobiasdefn}), it is easy to see that whenever $a \neq 0$, we
have
\begin{equation}\label{scaling}
    (aY)^* =_d a Y^*.
\end{equation}
Next, from \cite{goldsteinreinertzerobias}, if $Y$ is bounded by
some constant, then $Y^*$ is also bounded by the same constant, that
is,
\begin{equation}\label{boundedness}
    |Y|\leq c \quad \text{implies} \quad |Y^*| \leq c.
\end{equation}

Though Theorem \ref{maintheorem} may be invoked in the presence of
dependence, and for variables not expressed as sums, we
compare the performance of our bound to comparable results in the
literature whose application is limited to the case where $Y$ is the
sum of independent variables $X_1,\ldots,X_n$ with mean zero and
variances $\sigma_i^2 = \mbox{Var}(X_i) \in (0,\infty),
i=1,\ldots,n$. Letting $\sigma^2=\mbox{Var}(Y)$, following
\cite{goldsteinreinertzerobias}, one can form a zero biased coupling
of $Y$ to $Y^*$ by replacing the $I^{th}$ summand $X_I$ of $Y$ by a
random variable $X_I^*$, independent of the remaining summands, which
has the $I^{th}$ summand's zero bias distribution, where the index
$I$ has distribution $\mathbb{P}(I=i)=\sigma_i^2/\sigma$ and is
chosen independently of all else. When $|X_i|\leq c$ for all
$i=1,\ldots,n$,  by (\ref{boundedness}) this construction satisfies
$$|Y^*-Y|=|X_I^*-X_I| \le 2c,$$
and as the moment generating function of $Y$ exists everywhere in this case, using
the bound, say, (\ref{tailsseparate.2}) we obtain \bea
\label{ind.bernstein}
\mathbb{P}( Y \geq t)\leq \exp \left(-
\frac{t^2}{2 \sigma^2 + act} \right). \ena
with $a=4$. Perhaps the closest classical inequality to
 \eqref{ind.bernstein} that holds under the conditions above is the
one of Bernstein, see Corollary 2.11 of \cite{blm}, which yields
\eqref{ind.bernstein} with $a=2/3$. Though the constant of $2/3$ is
superior to $4$, our results are more general as they provide
concentration inequalities in the presence of dependence and for variables that need not be sums. Further, we also note
that the rate for large $t$ of the bound \eqref{tlogtbound} is
superior to the rate in \eqref{ind.bernstein} for any $a>0$.

The tail bounds in (\ref{tlogtbound}) can also be considered as a
version of Bennett's inequality for sums of independent random
variables. In the same setting as for \eqref{ind.bernstein} where
$Y$ is a sum of independent variables satisfying $|X_i| \leq c$,
Bennett's inequality, see Theorem 2.9 of \cite{blm}, provides the
tail bound \bea\label{bennett.independent}\mathbb{P}(Y \geq t) \leq
e^{t/c} \exp \left(- \frac{\sigma^2}{c^2}\left(1 +
\frac{ct}{\sigma^2} \right) \log \left(1 + \frac{ct}{\sigma^2}
\right) \right), \quad t \geq 0. \ena We note that in the case of
 independent summands, Bennett's inequality will in
general give better bounds than (\ref{tlogtbound}), but is again restricted to a sum of independent variables.


As the mean and variance pair $(\mu,\sigma^2)$ of a random variable
$Y$ may in general take on any value in $\mathbb{R} \times
(0,\infty)$, bounds for $Y$ expressed in terms of $\mu$, such as the
method of self bounding functions, see
\cite{mcdiarmid2006concentration}, and the use of size bias
couplings, see Ghosh and Goldstein (2011ab),
are not in general comparable to those of Theorem \ref{maintheorem}.
In particular, in Remark \ref{rem:chat}, while handling an example
involving dependent variables, we show how the bounds of Theorem
\ref{maintheorem}, expressed in terms of the variance, may be
superior to bounds expressed in terms of the mean.

\section{Hoeffding's permutation statistic}\label{hoeffding}
As discussed in the introduction, with $\pi$ a random permutation in
the symmetric group $S_n$ and $A=(a_{ij})_{1\leq i,j \leq n}$ an $n
\times n$ matrix with real entries, Hoeffding's statistic takes the
form
\begin{equation}\label{hoeffding-statistic}
Y = \sum_{i=1}^n a_{i\pi(i)}.
\end{equation}
Hoeffding's combinatorial central limit theorem \cite{hoe51} gives
conditions under which $Y$, properly centered and scaled, has an
asymptotic normal distribution. The rate of convergence of $Y$ to
its normal limit is well studied, see for instance
\cite{goldsteinbook} and references therein. Here we apply our main
results from Section \ref{mainresultssection} to obtain
concentration inequalities for $Y$ using zero bias couplings when $\pi$ is uniformly distributed over the symmetric group, and when its distribution is constant on conjugacy classes.

In the case where $\pi$ is uniform, when the rows of $A$ are monotone, or more generally, when they have the same relative order, the summand variables $\{a_{i \pi(i)}\}_{1 \le i \le n}$ are negatively associated and the Bernstein and Bennett inequalities hold, (\ref{ind.bernstein}) with $a=2/3$ and (\ref{bennett.independent}),
respectively, thus improving on the bound of Theorem \ref{hoeffdingtheorem} in this special case.
However, for both the uniform and constant conjugacy class distributions considered below it is easy to show that negative association does not hold in general.

\subsection{Uniform Distribution on Permutations}
Let $\pi$ be chosen uniformly over $S_n$. Letting
$$ a_{i\centerdot} = \frac{1}{n} \sum_{j=1}^n a_{ij}, \quad a_{\centerdot j} = \frac{1}{n} \sum_{i=1}^n a_{ij} \quad \text{and} \quad a_{\centerdot\centerdot} = \frac{1}{n^2} \sum_{i,j=1}^n a_{ij},$$
straightforward calculations show that the mean  of $Y$ is given by
$\mu_A = n a_{\centerdot\centerdot}$, and its variance by
\begin{equation}\label{varianceofhoeffding}
\sigma_A^2 = \frac{1}{n-1} \sum_{1 \le i,j \le n}
(a_{ij}^2-a_{i\centerdot}^2-a_{\centerdot j}^2 +
a_{\centerdot\centerdot}^2)=\frac{1}{n-1} \sum_{1 \le i,j \le
n}(a_{ij}-a_{i\centerdot}-a_{\centerdot
j}+a_{\centerdot\centerdot})^2.
\end{equation}
Further, let $||a||=\max_{1 \le i,j \le n}|a_{ij} - a_{i
\centerdot}|$. Avoiding trivialities, we assume $\sigma_A^2$ is non
zero.

\begin{theorem}\label{hoeffdingtheorem}
For $n \ge 3$ the bounds of Theorem \ref{maintheorem}  hold with $Y$
replaced by $Y-\mu_A$, $\sigma^2$ by $\sigma_A^2$ and $c = 8 ||a||$.
\end{theorem}

\begin{proof}
When $\pi$ has the uniform distribution over $S_n$, use of the
exchangeable pair approach of \cite{goldsteinreinertzerobias} for
constructing zero bias couplings, as applied in Theorem 2.1 of
\cite{goldsteinpattern}, see also Theorem 6.1 of
\cite{goldsteinbook}, yields a coupling of $(Y-\mu_A)^*$ to
$Y-\mu_A$ that satisfies $$|(Y-\mu_A)^* - (Y-\mu_A)| \leq 8 ||a||.$$
An application of Theorem \ref{maintheorem} now yields the claim. \bbox
\end{proof}



\begin{remark} \label{rem:chat} \cite{chatterjee} obtained the concentration bound
\begin{equation}\label{chatterjeebound}
\mathbb{P}(|Y-\mu_A| \geq t) \leq 2 \exp\left(- \frac{t^2}{4 \mu_A +
2t} \right)
\end{equation}
for the Hoeffding statistic $Y$ under the additional condition that
 $0 \le a_{i,j}
\le 1$ for all $i,j$. In this case we may take $||a||=1$ and $c=8$
in Theorem \ref{hoeffdingtheorem}, yielding
\begin{equation}\label{zbbound}
\mathbb{P}(|Y-\mu_A| \geq t) \leq 2 \exp \left(- \frac{t^2}{2
\sigma_A^2+16t}\right) .
\end{equation}
A simple computation shows that the bound (\ref{zbbound}) is smaller
than (\ref{chatterjeebound}) when $t \le (2\mu_A-\sigma_A^2)/7$.

When $a_{ij}, 1 \le i,j \le n$ are themselves independent random
variables with law ${\cal L}(U)$ having support in $[0,1]$, then
\beas \mathbb{E}[\sigma_A^2]= (n-1)\mbox{Var}(U) \le
(n-1)\mathbb{E}[U^2] < n \mathbb{E}[U] = \mathbb{E}[\mu_A], \enas
where the first equality follows by a calculation using the first
expression for the variance in (\ref{varianceofhoeffding}), then
applying $0 \le U \le 1$ to yield $\mathbb{E}[U^2] \le
\mathbb{E}[U]$ for use in the strict inequality. Hence if the array
entries behave as independent and identically distributed random
variables on $[0,1]$, the bound (\ref{zbbound}) will be
asymptotically preferred to (\ref{chatterjeebound}) everywhere.
Finally we note that the bound (\ref{tlogtbound}) further improves
on \eqref{zbbound}, as regards its asymptotic order in $t$.
\end{remark}

\subsection{Permutation distribution constant on cycle type}
We now consider Hoeffding's statistic \eqref{hoeffding-statistic}
when the distribution of $\pi$ is constant over cycle type. This
framework includes two special cases of note, one where $\pi$ is a
uniformly chosen fixed point free involution, considered by
\cite{goldsteinrinott} and \cite{ghosh}, having applications to
permutation testing in certain matched pair experiments, and the
other where $\pi$ has the uniform distribution over permutations
with a single cycle, considered by \cite{kolchin}, under the
additional restriction that $a_{ij}$ factors into a product $b_i d_j$. Bounds on the error
of the normal approximation to $Y$ when the distribution of $\pi$ is
constant over cycle type were derived in \cite{goldsteinpattern}.

We start by recalling some relevant definitions. For  $q=1,\ldots,n$
letting $f_q(\pi)$ be the number of $q$ cycles of $\pi$, the vector
\beas f(\pi) = (f_1(\pi),\ldots,f_n(\pi)) \enas is the cycle type of
$\pi$. For instance, the permutation $\pi=((1,3,7,5),(2,6,4))$ in
$S_7$ consists of one 4 cycle in which $1 \rightarrow 3 \rightarrow
7 \rightarrow 5 \rightarrow 1$, and one 3 cycle where $2 \rightarrow
6 \rightarrow 4 \rightarrow 2$, and hence has cycle type
$(0,0,1,1,0,0,0)$. We say the permutations $\pi$ and $\sigma $ are
of the same cycle type if $f(\pi) = f(\sigma)$, and that a
distribution $\mathbb{P}$  on $S_n$ is constant on cycle type if
$\mathbb{P}(\pi)$ depends only on $f(\pi), $ that is
$$\mathbb{P}(\pi) = \mathbb{P}(\sigma) \quad \text{whenever} \quad
f(\pi)=f(\sigma).$$

With $\mathbb{N}_0$ the set of non-negative integers, clearly a vector $f=(f_1,\ldots,f_n)$ is a
cycle type of a permutation in $S_n$ if and only if $f \in {\cal F}_n$ where
\beas
{\cal F}_n=\{(f_1,\ldots,f_n) \in \mathbb{N}_0^n: \sum_{i=1}^n if_i=n\}.
\enas
A special case of a distribution constant on
cycle type is one uniformly distributed over all permutations having
cycle type $f \in {\cal F}_n$, denoted ${\cal U}(f)$.  The situations where $\pi$
is uniformly chosen from the set of all fixed point free
involutions, and chosen uniformly from all permutations having a
single cycle, are both distributions of type $\mathcal{U}(f)$, the
first with $f=(0,n/2,0,\ldots,0)$ for even $n$ and the second with
$f = (0,0,\ldots,0,1).$

We consider distributions over $S_n$ having no fixed points with probability one, as is true for
the two special cases of most interest. Noting that under this
condition no expression of the form $a_{ii}$ appears in the sum
(\ref{hoeffding-statistic}), let \beas a_{io} = \frac{1}{n-2}
\sum_{j: j \not =i }^n a_{ij}
\qmq{and} a_{oo} = \frac{1}{(n-1)(n-2)} \sum_{i \not = j}a_{ij}.
\enas Under the symmetry condition $a_{ij}=a_{ji}$ for all $i \not
=j$, see \cite{goldsteinbook}, when $\pi_f$ has distribution
$\mathcal{U}(f)$ with $f_1=0$, the mean $\mu$ and variance
$\sigma_f^2=\mbox{Var}(Y_f)$ of the corresponding variable $Y_f$ for
$n \geq 4$ are given by \bea \label{variancecycle} \mu=(n-2) a_{oo}
\qmq{and} \sigma_f^2 = \left(\frac{1}{n-1} + \frac{2 f_2}{n(n-3)}
\right) \sum_{i \neq j} (a_{ij} -2a_{io} + a_{oo})^2. \ena For $n
\ge 4$, when $n$ is even and $\iota$ is the cycle type of a fixed
point free involution, then $\iota_k=(n/2){\bf 1}(k=2)$, and when
$f$ is the cycle type of a permutation without any fixed points or
two cycles, such as is the case for one long cycle, then $f_2=0$,
and the variance in (\ref{variancecycle}) specializes, respectively,
to \bea \label{variance.fpfi} \sigma_\iota^2
=\frac{2(n-2)}{(n-1)(n-3)} \sum_{i \neq j} (a_{ij} -
2a_{io}+a_{oo})^2 \qmq{and} \sigma_f^2 = \frac{1}{n-1} \sum_{i \neq
j} (a_{ij} -2a_{io} + a_{oo})^2. \ena Let also \bea
\label{def.c.conj} a_o = \max_{i \neq j} |a_{ij}-2a_{io} + a_{oo}|.
\ena

Lemma 6.5 of \cite{goldsteinbook} shows that when a distribution
$\mathbb{P}$ on $S_n$ is constant on cycle type then it can be
represented as the mixture of the uniform distributions ${\cal
U}(f)$ for $f \in {\cal F}_n$,
\beas
\mathbb{P}=\sum_{f \in {\cal F}_n} \rho_f {\cal U}(f) \qmq{where}
\rho_f = \mathbb{P}(f(\pi)=f). \enas In particular, when the
distribution of $\pi$ is constant on cycle type and $Y$ is given by
\eqref{hoeffding-statistic} then
\bea\label{P.is.mix} \mathbb{E}[Y]=\mu, \quad
\mbox{and} \quad \mathbb{P}(Y-\mu \ge t) = \sum_{f \in {\cal F}_n} \rho_f \mathbb{P}(Y_f-\mu \ge t).
\ena
Hence, bounds for $Y$ are implied by those for $Y_f, f \in {\cal F}_n$.

\begin{theorem} 
Let $n \ge 5$ and $(a_{ij})_{i,j=1}^n$ be an array of real numbers
satisfying $a_{ij} = a_{ji}$. When $\pi$ is a uniformly chosen fixed point free involution, the bounds of Theorem \ref{maintheorem}  hold with $Y$ replaced by $Y-\mu$, $\sigma^2$ replaced by
$\sigma_\iota^2$ of (\ref{variance.fpfi}), and $c$ replaced by $24a_o$ with $a_o$ of
(\ref{def.c.conj}). When the distribution of $\pi$ is constant on cycle type and has no fixed points or two cycles with probability one, the bounds of Theorem \ref{maintheorem} hold with $Y$ replaced by $Y_f-\mu$, $\sigma^2$ replaced by $\sigma_f^2$ of (\ref{variance.fpfi}), and $c$ replaced by $40a_o$.
\end{theorem}

\begin{proof} When $\pi_\iota$ has the $\mathcal{U}(\iota)$ distribution, using the exchangeable
pair approach of \cite{goldsteinreinertzerobias}, the construction in Lemma 6.10 of \cite{goldsteinbook}
provides a zero biased
coupling for $Y_\iota-\mu$ that satisfies
\bea \label{24}
|(Y_\iota-\mu)^* - (Y_\iota-\mu)| \leq 24 a_o.
\ena
Theorem \ref{maintheorem} now obtains to yield the first claim. Similarly, Theorem 2.2 of \cite{goldsteinpattern} shows that the constant 24 in \eqref{24} can be replaced by 40 when $Y_\iota$ is replaced by $Y_f$ for any $f$ satisfying $f_1=f_2=0$. The final claim of the theorem is now obtained by applying the last equality of \eqref{P.is.mix}, noting that the variances and coupling constants for all such $Y_f$, and therefore the upper bounds on $\mathbb{P}(Y_f-\mu \ge t)$ produced by Theorem \ref{maintheorem}, are identical, and that $\rho_f$ sums to one. \bbox
\end{proof}

We remark that bounds for the general situation where the distribution of $\pi$ is constant on cycle type, without fixed points, can be obtained using \eqref{P.is.mix} in the same fashion, yielding a weighted sum of bounds of the two types appearing in Theorem \ref{maintheorem}.

\section{Proof of Main Result} \label{proofs}

\emph{Proof of Theorem \ref{maintheorem}:} Let
$m(s)=\mathbb{E}[e^{sY}]$ and $m^*(s)=\mathbb{E}[e^{sY^*}]$. When
$Y^*-Y \le c$ for all $s \ge 0$ then \bea \label{take.expz}
m^*(s)=\mathbb{E}[e^{s Y^*}] = \mathbb{E}[e^{s (Y^*-Y)} e^{s Y}] \le
\mathbb{E}[e^{cs} e^{s Y}] = e^{cs}m(s). \ena In particular when
$m(s)$ is finite then so is $m^*(s)$.

Part (a). If $m(s)$ exists in an open interval containing $s$ we may
interchange expectation and differentiation at $s$ to obtain \bea
\label{mprimeestim} m'(s)=\mathbb{E}[Y e^{s Y}]=\sigma^2
\mathbb{E}[s e^{s Y^*}] = \sigma^2 s m^*(s), \ena where we have
applied the zero bias relation (\ref{zerobiasdefn}) to yield the
second equality.

We first prove (\ref{tailsseparate.2}). Starting with the well known
inequality $1-x \le e^{-x}$, holding for all $x \ge 0$, we obtain
\beas e^x \le \frac{1}{1-x} \qmq{for $x \in [0,1)$.} \enas Hence,
for $\theta \in (0,1/c)$ and $0 \le s \le \theta$ we have \beas
m^*(s)=\mathbb{E}[e^{s Y^*}]=\mathbb{E}[e^{s (Y^*-Y)}e^{s Y}] \le
e^{s c}m(s) \le \frac{1}{1-s c}m(s). \enas Using the identity
(\ref{mprimeestim}) to express $m^*(s)$ in terms of $m'(s)$ we
obtain \beas m'(s) \le \frac{\sigma^2 s}{1-s c}m(s). \enas Dividing
both sides by $m(s)$, integrating over $[0,\theta]$ and using that
$m(0)=1$ we obtain \beas \log m(\theta) = \int_0^\theta
\frac{m'(s)}{m(s)}ds \le \frac{\sigma^2}{1-\theta c}\int_0^\theta
sds =\frac{\sigma^2 \theta^2}{2(1-\theta c)}, \enas and
exponentiation yields \beas m(\theta) \le \exp \left( \frac{\sigma^2
\theta^2}{2(1-\theta c)} \right). \enas

As (\ref{tailsseparate.2}) holds trivially for $t=0$ consider $t>0$
and apply Markov's inequality to obtain \beas \mathbb{P}(Y \ge t) =
\mathbb{P}(e^{\theta Y} \ge e^{\theta t} )\le e^{-\theta t}
m(\theta) \le \exp \left(-\theta t + \frac{\sigma^2
\theta^2}{2(1-\theta c)} \right). \enas Setting $\theta=t/(\sigma^2
+ ct)$, and noting this value lies in the interval $(0,1/c)$,
(\ref{tailsseparate.2}) follows.

If now $Y-Y^* \le c$ then letting $X=-Y$ we see from (\ref{scaling})
with $a=-1$ that $X^*=-Y^*$ has the $X$-zero bias distribution, and
applying (\ref{tailsseparate.2}) to $X$ yields the claimed left tail
inequality.

Turning to (\ref{tailsseparate}), by the convexity of the
exponential function we have
\beas 
\frac{e^y - e^x}{y-x} = \int_0^1 e^{t y +(1-t) x} dt \leq \int_0^1
(t e^y + (1-t) e^x) dt= \frac{e^y + e^x}{2} \qmq{for all $x \neq
y$,} \enas and hence \beas e^y - e^x \leq \frac{|y - x|(e^y +
e^x)}{2} \qmq{for all $x$ and $y$.} \enas

Hence, when $|Y^*-Y| \leq c$, for all $\theta \in (0,2/c)$ and $0
\le s \le \theta$,
\beas 
e^{s Y^*} - e^{s Y} \leq \frac{|s (Y^*-Y)| (e^{s Y^*} + e^{s Y})}{2}
\leq \frac{c s}{2} (e^{s Y^*} + e^{s Y}). \enas Taking expectation
yields
\beas 
m^*(s) - m(s) \leq \frac{c s}{2} (m^*(s) + m(s)), \qmq{hence}
    m^*(s) \leq \left(\frac{1+c s /2}{1-c s /2}
    \right) m(s),
\enas and now relation (\ref{mprimeestim}) gives that
\beas 
m'(s) \leq \sigma^2 s \left(\frac{1+c s /2}{1 -c s /2} \right) m(s).
\enas Following steps similar to the ones above, we obtain
\beas 
m(\theta) \leq \exp \left(\frac{\alpha \sigma^2 \theta^2}{1-c \theta
/2} \right) \qmq{for $\alpha=5/6$ and all $\theta \in (0,2/c)$.}
\enas

As the result holds trivially for $t=0$, fix $t>0$ and argue as
before using Markov's inequality to obtain
\begin{eqnarray*}
  \mathbb{P}(Y \geq t)
  \leq \exp \left(-\theta t + \frac{\alpha \sigma^2 \theta^2}{1-c \theta /2}
    \right) \qmq{for all $\theta \in (0,2/c)$.}
\end{eqnarray*}
Letting $\theta = 2t/(4\alpha \sigma^2 + ct)$, and noting that this
value lies in $(0,2/c)$, we obtain the asserted right tail
inequality (\ref{tailsseparate}). Replacing $Y$ by $-Y$ as before
now demonstrates the remaining claim.

Proof of (b). For any $s \in [0,\theta)$ such that $m(\theta)$
exists, by \eqref{take.expz} we have \beas
m'(s)=\mathbb{E}[Ye^{s Y}]=\sigma^2 s \mathbb{E}[e^{s Y^*}] =
\sigma^2 s m^*(s)\le \sigma^2 s e^{cs} m(s), \qmq{so that} ( \log
m(s) )' \le \sigma^2 s e^{cs}. \enas Integrating over $[0,\theta]$
and using that $m(0)=1$ we obtain \beas \log(m(\theta)) \le
\frac{\sigma^2}{c^2}\left(e^{c \theta}\left(c \theta -1 \right)+1
\right) \enas and exponentiation yields \beas m(\theta) \le
\exp\left(\frac{\sigma^2}{c^2}\left(e^{c \theta}\left(c \theta -1
\right)+1 \right)\right). \enas Applying Markov's inequality as
before, \beas
\mathbb{P}(Y \ge t)
\le \exp \left( -\theta t + \frac{\sigma^2}{c^2}\left(e^{c
\theta}\left(c \theta -1 \right)+1 \right) \right). \enas For $t>e$
letting $\theta=(\log t - \log \log t)/c$ we obtain the first claim of \eqref{tlogtbound} by
\begin{multline*}
\mathbb{P}(Y \ge t) \le \exp\left( -\frac{t}{c}\left( \log t - \log
\log t\right) +
\frac{\sigma^2}{c^2} \left( \frac{t}{\log t} \left(\log t - \log \log t -1 \right) + 1 \right) \right)\\
\le \exp\left( -\frac{t}{c}\left( \log t - \log \log t  -
\frac{\sigma^2}{c}\right)\right).
\end{multline*}
The second claim follows by the inequality $(\log t)/2 \ge \log \log
t$ for all $t>1$. The left tail bound follows as in (a). \bbox

\end{document}